\documentclass{ps}

\usepackage{amssymb,amsbsy,amsmath,amscd,amsthm,amsfonts}
\usepackage{cite}
\usepackage[utf8]{inputenc}
\usepackage{enumerate,hyperref}
\usepackage{dsfont}

\hypersetup{colorlinks=true, citecolor=blue 
}

\newtheorem{theorem}{Theorem}
\newtheorem{condition}{Assumption}

\newtheorem{lemma}{Lemma}
\newtheorem{remark}{Remark}

\newtheorem{corollary}{Corollary}

\newcommand{\R}{\mathbb R}
\newcommand{\PP}{\mathbb P}
\newcommand{\E}{\mathbb E}
\newcommand{\VV}{\mathbb V}
\newcommand*\diff{\mathop{}\!\mathrm{d}}

\newcommand{\DS}{\displaystyle}

\begin{document}

\title{A change-point problem and inference for segment signals}

\author{Victor-Emmanuel Brunel}\address{Massachusetts Institute of Technology, Department of Mathematics}

%
%
\begin{abstract}
  We address the problem of detection and estimation of one or two change-points in the mean of a series of random variables. We use the formalism of set estimation in regression: To each point of a design is attached a binary label that indicates whether that point belongs to an unknown segment and this label is contaminated with noise. The endpoints of the unknown segment are the change-points. We study the minimal size of the segment which allows statistical detection in different scenarios, including when the endpoints are separated from the boundary of the domain of the design, or when they are separated from one another. We compare this minimal size with the minimax rates of convergence for estimation of the segment under the same scenarios. The aim of this extensive study of a simple yet fundamental version of the change-point problem is twofold: Understanding the impact of the location and the separation of the change points on detection and estimation and bringing insights about the estimation and detection of convex bodies in higher dimensions.
\end{abstract}

\begin{resume} On s'intéresse aux problèmes de détection et d'estimation d'un ou de deux points de rupture dans une suite de variables aléatoires. Le formalisme utilisé est celui de l'estimation d'ensembles dans le cadre de la regression: chaque observation est accompagnée d'une variable binaire indiquant si l'observation est à l'intérieur d'un segment inconnu, et cette variable est observée avec un bruit additionnel. Les extremités du segment correspondent aux points de rupture. On caractérise la taille minimale du segment permettant sa détection, dans différents scénarios: en particulier, lorsque ses extrémités sont suffisamment eloignées du bord du domaine, l'une de l'autre. Ensuite, on compare cette taille minimale aux taux minimax de convergence pour l'estimation du segment, dans ces mêmes scénarios. L'objectif de l'étude exhaustive de cette version élémentaire mais fondamentale du problème de rupture est double: on cherche à comprendre l'impact de la localisation des points de rupture sur leur detection et leur estimation, ainsi qu'à élaborer de nouvelles idées pour l'estimation et la détection d'ensembles convexes en plus grande dimension.

\end{resume}

\maketitle

\section{Introduction}
	
\subsection{The change-point problem}

Change-point problems have been studied extensively, especially in time series analysis, where the goal is to detect or estimate breakpoints in the distribution of ordered observations. The breakpoints can occur in the mean \cite{KTlectureNotes1993,Korostelev1986,Raimondo1998,ParkKim2004,Ninomiya2005,Wu2005,Betken2016} and/or in the variance \cite{ChenGupta2004,SonKim2005}, in a location parameter \cite{AntochHuvskova2000}, in the tail of the distribution \cite{GadeikisPaulauskas2005}, in a general parameter of the distribution \cite{LeeHaNa2003}, or the whole distribution can change at the breakpoints \cite{BorovkovLinke2004}. For more details on the account of change-point problems in time series analysis, we refer to (see \cite{ShaoZhang2010, Betken2016} and the references therein. Here, we focus on change-points in ordered data that do not necessarily come from time series. Then, if the change-points occur in the mean, this problem can be stated in terms of inference on breakpoints of a regression function. A general problem has been addressed in \cite{FrickMunkSieling2013} and includes cases where breakpoints occur in the mean: A sample $Y_1,\ldots,Y_n$ is observed, where $Y_i$ has a density $f(\cdot,\phi(\frac{i}{n}))$ with respect to a given measure, for $i=1,\ldots,n$. The map $f$ belongs to some parametric class of densities. The real valued function $\phi$ is assumed to be piecewise constant on $[0,1]$, with a finite number $K$ of jumps (called the change points, or breakpoints), where $K$ is not necessarily known. It is shown that at least one change-point can be detected consistently. When $f(\cdot,\mu)$ is the Gaussian density with mean $\mu$ and given variance $\sigma^2>0$, the problem was addressed in \cite{Lebarbier2003}. That problem can also be interpreted as estimating the jumps of a regression function. Its simplest form, with only one change-point, reads as follows:

\begin{equation}
	\label{ChangePoint}
	Y_i=\mathds 1(X_i\leq \theta)+\xi_i, i=1,\ldots,n,
\end{equation}
where $X_1,\ldots,X_n$ are given numbers (possibly random) in $[0,1]$, $\xi_1,\ldots,\xi_n$ are i.i.d. random variables independent of the $X_i$'s and $\theta\in [0,1]$ is the change-point (or breakpoint).

Model \eqref{ChangePoint} was studied in \cite[Sec. 1.9]{KTlectureNotes1993} and a continuous-time version was addressed in \cite{Korostelev1986}, where the aim is to estimate the breakpoint $\theta$. 
In the continuous-time version, Korostelev \cite{Korostelev1986} focused on a more general framework, where the regression function has a jump but is not necessarily an indicator function. In \cite{Korostelev1986,KTlectureNotes1993}, the change-point $\theta$ is estimated with an expected accuracy of order $1/n$. Ibragimov and Khasminskii \cite{Ibragimov1975,Ibragimov1984} defined a consistent estimator of the discontinuity point of a regression function, with rate $1/n$ as well. However, a key assumption in all these works is that the change-point is separated from the boundaries of the domain: $h\leq \theta \leq 1-h$, for some $h\in (0,1/2)$. 

The separation assumption from the boundaries of the domain is also made in higher dimensional problems. For instance, in \cite{KorostelevTsybakov1992} and \cite[Chap 3]{KTlectureNotes1993}, a boundary fragment (which plays the higher dimensional role of $[0,\theta]$ in \eqref{ChangePoint}) is estimated, under the assumption that its edge function is uniformly separated from $0$ and $1$.

As part of this work, we prove that this separation assumption is only technical and that the estimation rate $1/n$ (with no extra logarithmic factor) is achieved without separation in Model \eqref{ChangePoint}. Our focus is a one dimensional model, where the number of breakpoints is known and is either one or two, which allows an interpretation of the model in terms of detection and estimation of segments.

More generally, we consider the following statistical model: 
	\begin{equation}
		\label{Model}
		Y_i=\mathds 1(X_i\in G)+\xi_i, i=1,\ldots,n,
	\end{equation}
	where $\mathds 1(\cdot)$ is the indicator function.
	The collection $\mathcal X=\{X_1,\ldots,X_n\}$ is called the design and it is observed, as well as the labels $Y_1,\ldots,Y_n$.
	The unknown set $G$ is a segment on $[0,1]$ and the noise terms $\xi_i$ are unobserved i.i.d. random variables, independent of the design. 

Throughout the paper, we assume that the noise terms $\xi_1,\ldots,\xi_n$ from Model \eqref{Model} are subgaussian, i.e., satisfy 
\begin{equation}\label{subgauss}
		\E\left[e^{u\xi_i}\right]\leq e^{\sigma^2 u^2/2}, \quad \forall u\in\R, \quad i=1,\ldots,n,
\end{equation}
	for some positive constant $\sigma>0$. This constant need not be known. Note that \eqref{subgauss} implies that the noise terms have mean zero. If they are centered Gaussian random variables, then they satisfy \eqref{subgauss} with $\sigma^2=\textsf{Var}(\xi_1)$. 
	
  Since the design and the noise are assumed to be independent, reordering the $X_i$'s does not affect the model. Indeed, there exists a reordering $\{i_1,\ldots,i_n\}$ of $\{1,\ldots,n\}$, such that $X_{i_1}\leq\ldots\leq X_{i_n}$. The random indices $i_1,\ldots,i_n$ are independent of the noise, therefore, the new noise vector $(\xi_{i_1},\ldots, \xi_{i_n})$ has the same distribution as $(\xi_1,\ldots, \xi_n)$. Hence, we assume from now on that $\mathcal X$ is the reordering of a preliminary design, i.e., $X_1\leq\ldots\leq X_n$ almost surely, without loss of generality. We distinguish two types of designs:
	\begin{enumerate}
		\item[(DD)] Deterministic, regular design: $X_i=i/n, i=1,\ldots,n$;
		\item[(RD)] Random, uniform design: the variables $X_i, i=1,\ldots,n$, are the reordering of i.i.d. uniform random variables in $[0,1]$.
	\end{enumerate}
Other designs are considered in the literature (e.g., see\cite[Sec. 1.9]{KTlectureNotes1993}) but we prefer to restrict ourselves to the designs (DD) and (RD), which yield straightforward extensions to other types of designs.
	Note that Model \eqref{Model} can also be interpreted as a nonlinear regression, where we do inference on the support $G$ of the regression function. Actually, it is the one dimensional version of the model studied in \cite{Brunel2013}, where $G$ plays the role of a $1$-dimensional convex body. In \cite{Brunel2013}, it is explained that hardness of estimation of $G$, in Model \eqref{Model}, can be explained by two factors. The first factor is the complexity of the class of possible candidates $G$ and the second one is that $G$ might be too small to be detected by any statistical procedure. In the present case, $G$ is a segment, hence the corresponding class is parametric and one can hope to estimate $G$ at the fast speed $1/n$, up to a positive multiplicative constant. However, we show that surprisingly, if the class contains arbitrarily short segments, the speed of estimation may be deteriorated. We try to understand what type of scenarios allow estimation of $G$ at the fast rate $1/n$ and under what other scenarios the best rate of estimation of $G$ is significantly worse than $1/n$, in a minimax sense. In particular, we study the two following assumptions on $G$, where $\mu\in (0,1)$ is fixed throughout the paper.
	
\begin{condition}\label{A1}
	$G=[0,\theta]$, for some unknown number $0\leq \theta \leq 1$.
\end{condition}

This assumption carries information about the location of $G$. In terms of the change-point problem, it implies that there is only one change-point and the problem becomes equivalent to \eqref{ChangePoint}, where no separation of $\theta$ from $0$ and $1$ is imposed. 

The second assumption implies that $G$ is not too short, i.e., that the two change-points are separated from one another:

\begin{condition}\label{A2}
	$|G|\geq \mu$, where $\mu\in(0,1)$ is a known positive number.
\end{condition}	

Here, $|G|$ is the length of the segment $G$.

Model \eqref{Model} with random design has been addressed in higher dimensions in \cite{Brunel2013}, where $G$ is assumed to be a convex polytope with fixed number of vertices. In that work, there were no assumptions of the type of \ref{A1} or \ref{A2} and the rate of estimation was $(\ln n)/n$, which, we believe, would become faster under similar assumptions as Assumptions \ref{A1} or \ref{A2}. Further details on this account are given in the discussion in Section \ref{ConclusionDiscussion}.


The detection problem consists of testing whether $G=\emptyset$ in Model \eqref{Model}, i.e., whether there is no change-point. It is addressed, for example, in \cite{ChanWalther2013}, where, unlike here, the authors do not assume the strength of the signal to be known:
	\begin{equation}
		\label{Model2}
		Y_i=\delta\mathds 1(X_i\in G)+\xi_i, i=1,\ldots,n,
	\end{equation}
where $\delta$ is an unknown positive number. For the signal to be detectable, there should be a tradeoff between its length $|G|$ and its strength $\delta$. Intuitively, if $\delta$ is small, then the set $G$ should be big enough and conversely, if $\delta$ is large, the set $G$ can be short and the signal still be detected. In that framework, testing the presence of a signal reduces to decide whether $\delta=0$, which makes the problem different from ours, where we know $\delta$ and impose $\delta=1$. In \cite{ChanWalther2013}, the authors mainly study the power of two tests under design (DD): the scan - or maximum - likelihood ratio and the average likelihood ratio. These are compared in two regimes: signals of small scales, i.e. $|G|\longrightarrow 0$, when $n\rightarrow\infty$ and signals of large scales, i.e. $\displaystyle{\underset{n\rightarrow\infty}{\operatorname{liminf}} \text{ } |G|>0}$. It is proved that if $\displaystyle{\delta\sqrt{n|G|}\geq\sqrt{2\ln\frac{1}{|G|}}+b_n}$, for some sequence $b_n$ such that $b_n\longrightarrow \infty$, then there is a test with asymptotic power 1. For fixed $\delta$, Chan and Walther's condition implies that $|G|$ must be of order at least $(\ln n)/n$. We prove a similar condition when $\delta=1$ is known. This means that knowing $\delta$ does not make the detection problem easier and the logarithmic factor in Chan and Walther's condition is not due to adaptation to $\delta$.

 Note that $\delta\sqrt{|G|}$ is exactly the $L^2$-norm of the signal. In \cite{LepskiTsybakov2000}, signals of unknown shape but known smoothness were considered. The authors test whether the observations are pure noise and give exact minimax separation rates in terms of the $L^2$-norm of the signal. Detection is harder in that framework, because unlike in Model \eqref{Model} or \eqref{Model2}, the shape of the signal is unknown and only its smoothness is known. This is why the separation rates are larger than those corresponding to models \eqref{Model} and \eqref{Model2}, in the sense that they allow less freedom for the size of the signal to be detected. A similar detection problem has also been studied in \cite{CaiYuan2014}, where three cases are considered: either the shape of the signal (up to an affine transform), or its smoothness is known, or nothing is known. However, this problem is different from ours, since we are only concerned with the location of the signal, not the signal itself. This is why, and also for the sake of simplicity, we only deal with signals of known shape and amplitude in the present work.



\subsection{Definitions and notation}

	If $G$ is a segment in $[0,1]$, we denote by $\PP_G$ the joint probability measure of the observations $(X_1,Y_1),\ldots,(X_n,Y_n)$ that satisfy Model \eqref{Model} and by $\E_G$ and $\VV_G$ the corresponding expectation and variance operators. We may omit the subscript $G$ and write only $\PP$, $\E$ or $\VV$ if there is no ambiguity. 
	
	If $G_1$ and $G_2$ are two segments in $[0,1]$, we denote by $G_1\triangle G_2$ their symmetric difference. Its Lebesgue measure $|G_1\triangle G_2|$ is also called the Nykodim distance between $G_1$ and $G_2$.

	An estimator of $G$ is a segment (possibly empty) $\hat G_n$ of $[0,1]$, whose construction depends on the observations.
We measure the accuracy of an estimator $\hat G_n$ in a minimax framework. The risk of $\hat G_n$ on a class $\mathcal C$ of segments is defined as
\begin{equation*}
	\mathcal R_n(\hat G_n ; \mathcal C) = \sup_{G\in\mathcal C}\mathbb E_G[|G\triangle\hat G_n|].\tag{$*$}
\end{equation*}
The rate (a sequence which depends on $n$) of an estimator on a class $\mathcal C$ is the speed at which its risk converges to zero when the number $n$ of available observations tends to infinity. The minimax risk on a class $\mathcal C$, when $n$ observations are available, is defined as
\begin{equation*}
	\mathcal R_n(\mathcal C)=\inf_{\tilde G_n} \mathcal R_n(\hat G_n ; \mathcal C), \tag{$**$}
\end{equation*}
where the infimum is taken over all estimators $\tilde G_n$ depending on $n$ observations. If $\mathcal R_n(\mathcal C)$ converges to zero, we call the minimax rate of convergence on the class $\mathcal C$ the speed at which $\mathcal R_n(\mathcal C)$ tends to zero.


For the detection problem, let $h$ be a positive number, that may depend on $n$. We consider the following hypotheses:
$$
\begin{cases}
	H_0 : G=\emptyset \mbox{ (the null hypothesis)}\\
	H_1 : |G|\geq h \mbox{ (the alternative hypothesis)}.
\end{cases}
$$
The performance of a test $\tau_n\in\{0,1\}$ on a class $\mathcal C$ is measured by the sum of its type one and two errors, i.e.,  
$$\gamma_n(\tau_n,\mathcal C)=\PP_\emptyset\left[\tau_n=1\right]+\sup_{G\in\mathcal C,|G|\geq h}\PP_G\left[\tau_n=0\right].$$
We say that $\tau_n$ is consistent on the class $\mathcal C$ if and only if $\gamma_n(\tau_n,\mathcal C)\longrightarrow 0$ when $n\rightarrow\infty$. We call the separation rate on the class $\mathcal C$ any sequence of positive numbers $r_n$ such that:
\begin{itemize}
	\item if $\displaystyle{\frac{h}{r_n} \underset{n\rightarrow\infty}{\longrightarrow}\infty}$, then there exists a consistent test on $\mathcal C$ and
	\item if $\displaystyle{\frac{h}{r_n} \underset{n\rightarrow\infty}{\longrightarrow} 0}$, then no test is consistent test on $\mathcal C$.
\end{itemize}

We define three different classes of segments: 
\begin{description}
	\item[- ] $\mathcal S=\left\{[a,b] : 0\leq a\leq b\leq 1\right\}$ is the class of all segments on $[0,1]$,
	\item[- ] $\mathcal S_0=\left\{[0,\theta] : 0\leq \theta\leq 1\right\}$ is the class of all segments on $[0,1]$, satisfying Assumption \ref{A1},
	\item[- ] $\mathcal S(\mu)=\left\{G\in\mathcal S : |G|\geq \mu\right\}$ is the class of all segments on $[0,1]$, satisfying Assumption \ref{A2}.
\end{description}

For two real valued sequences $A_n$ and $B_n$ and a parameter $\vartheta$, which may be multidimensional, we will write $A_n\asymp_\vartheta B_n$ when there exist positive constants $c(\vartheta)$ and $C(\vartheta)$, which depend on $\vartheta$ only, such that $c(\vartheta)B_n\leq A_n\leq C(\vartheta)B_n$, for $n$ large enough. If we put no subscript under the sign $\asymp$, this means that the corresponding constants are universal, i.e., do not depend on any parameter.

\subsection{Outline and Contributions}

In Section \ref{Sec:Detection}, we tackle the detection problem. We prove that the separation rates are at least $1/n$ for the class $\mathcal S_0$ and $(\ln n)/n$ for $\mathcal S$ and that these are exactly the separation rates when the noise is Gaussian (Theorem \ref{TestThm}). In Section \ref{Sec:Estimation}, we estimate the unknown set and we prove minimax rates of on the classes $\mathcal S$, $\mathcal S_0$ and $\mathcal S(\mu)$. Specifically, we show that the breakpoint $\theta$ in \eqref{ChangePoint} can be estimated at the speed $1/n$ without imposing separation from 0 and 1: We prove that the minimax rate of convergence on the class $\mathcal S_0$ is $1/n$ (Theorem \ref{minimaxCP}). We also show that recovering a general segment (or, equivalently, two change-points) with a uniform risk of order $1/n$ is possible only if it is known \textit{a priori} that the unknown set is not too small (Theorem \ref{minimaxriskgen}). Otherwise, a logarithmic factor appears in the minimax rate (Theorem \ref{minimaxlarge}). For $G\in\mathcal S(\mu)$, we define a two step estimator inspired by \cite{Korostelev1986}. We first locate $G$ with high probability using the first half of the sample and we apply the same ideas as in the class $\mathcal S_0$ to estimate the endpoints of $G$ separately, using the second half of the sample.

In Section \ref{ConclusionDiscussion} we draw conclusions, followed by a discussion about possible extensions and all the proofs are deferred to Section \ref{Sec:Proofs}.

\section{Detection of a set} \label{Sec:Detection}

	In this section, we find separation rates for testing $H_0 : G=\emptyset$ against $H_1 : |G|\geq h$, when $G$ belongs to some classes of segments.

The idea for the class $\mathcal S_0$ is the following. Under $H_1$, since $G\in\mathcal S_0$, we necessarily have that $[0,h]\subseteq G$. Therefore, we check among those pairs $(X_i,Y_i)$ for which $X_i\leq h$ if there are sufficiently many $Y_i$'s that are large, e.g. larger than $1/2$. Let $N=\max\{i=1,\ldots,n : X_i\leq h\}=\#\left(\mathcal X\cap [0,h]\right)$, where $\#$ stands for cardinality. Let $S$ be the following test statistic:
	\begin{equation*}
		S=\#\{i=1,\ldots,N : Y_i\leq \frac{1}{2}\}.
	\end{equation*}
If the alternative hypothesis is true, i.e. if $|G|\geq h$, then all the $X_i$'s, $i\leq N$, fall inside $G$ and the corresponding $Y_i$'s should not be too small. We define the test $\DS T_n^0=\mathds 1(S\leq cN)$, where $c$ is any number strictly between $\PP[\xi_1\leq -1/2])$ and $\PP[\xi_1\leq 1/2]$. Note that this definition requires some knowledge about the noise distribution. In many cases, it may be reasonable to assume that the noise is symmetric, hence to take $c=1/2$. However, in general, it is not clear how to calibrate $c$ if no information about the noise is available.

For the class $\mathcal S$, we consider a pseudo likelihood ratio test (which is a likelihood ratio test if the noise is Gaussian). For $G\in\mathcal S$, let $R(G)=\sum_{i=1}^n Y_i\mathds 1(X_i\in G)-\frac{\#(\mathcal X\cap G)}{2}$, and let $R=\sup_{|G|\geq h}R(G)$. Under the alternative hypothesis, $R$ should be quite large, hence, we define the test $T_n^1=\mathds 1(R\geq 0)$. Note that this test reduces to scanning the interval $[0,1]$ and seeking for a large enough quantity of successive observations with large $Y_i$, which can be done in a quadratic number (in $n$) of steps.

	\begin{theorem} \label{TestThm}
		Let Model \eqref{Model} hold. 
		\begin{enumerate} 
			\item Assume that the design is (DD) or (RD) and that the noise satisfies:
						$$\PP[\xi_1\leq -1/2] < \PP[\xi_1\leq 1/2].$$ 
						Then, if $nh\longrightarrow\infty$, the test $T_n^0$ is consistent, i.e. $\gamma_n(T_n^0,\mathcal S_0)\longrightarrow 0$. In addition, if the noise is Gaussian, then $r_n=1/n$ is a separation rate on the class $\mathcal S_0$.
			\item Assume that the design is (DD) or (RD). Then, if $nh/\ln n\longrightarrow \infty$, the test $T_n^1$ is consistent, i.e. $\gamma_n(T_n^1,\mathcal S)\longrightarrow 0$. In addition, if the noise is Gaussian, then $r_n=(\ln n)/n$ is a separation rate on the class $\mathcal S$.
		\end{enumerate}
	\end{theorem}

In the next section, we show that the separation rates given in Theorem \ref{TestThm} are equal to the minimax rates of convergence on the corresponding classes.

\section{Estimation of a segment} \label{Sec:Estimation}

\subsection{A least square estimator}\label{Least square estimator}

	Let Model \eqref{Model} hold. For $G'\in\mathcal S$, let $A_0(G')=\sum_{i=1}^n\left(Y_i-\mathds 1(X_i\in G')\right)^2$ be the sum of squared errors. In order to estimate the true and unknown set $G$, we find a random set $\hat G_n$ which minimizes $A_0(G')$, among all possible candidates $G'$. Note that minimizing $A_0(G')$ is equivalent to maximizing 
\begin{equation}
	\label{LSECriterion}
	A(G')=\sum_{i=1}^n(2Y_i-1)\mathds 1(X_i\in G').
\end{equation}
Denote by $S=\{i=1,\ldots,n : X_i\in G\}$ and by $S'=\{i=1,\ldots,n : X_i\in G'\}$, for some $G'\in\mathcal S$. If we write $A(G')$ as a function of $S'$, \eqref{LSECriterion} becomes
	\begin{align*}
		A(S') & = \sum_{i\in S'}(2Y_i-1) \\
		& = \sum_{i\in S'}\left(2\mathds 1(X_i\in G)+2\xi_i-1\right) \\
		& = 2\#(S\cap S')-\#S'+2\sum_{i\in S'}\xi_i,
	\end{align*}
so, 
	\begin{equation}
		\label{GausProc}
		A(S')-A(S) = -\#(S\triangle S')+2\left(\sum_{i\in S'\backslash S}\xi_i-\sum_{i\in S\backslash S'}\xi_i\right).
	\end{equation}
We call a set $S'\subseteq\{1,\ldots,n\}$ convex if and only if it is of the form $\{i,\ldots,j\}$, for some $1\leq i\leq j\leq n$. It is clear that if a convex subset $S'$ of $\{1,\ldots,n\}$ maximizes $A(S')-A(S)$, then the segment $G'=[X_{\min S'},X_{\max S'}]$ maximizes $A(G')$.

\subsection{Estimation of one change-point} \label{Sec:OneChangePoint}

\paragraph{Under the deterministic design (DD)}

	Let Model \eqref{Model} hold, with design (DD). Assume that $G=[0,\theta]\in\mathcal S_0$, where $\theta\in [0,1]$. Let us make one preliminary remark. For any estimator $\hat G_n$ of $G$, the random segment $\tilde G_n= [0,\sup \hat G_n]$ performs better than $\hat G_n$, since $|\tilde G_n\triangle G|\leq|\hat G_n\triangle G|$ almost surely. Therefore, it is sufficient to only consider estimators of the form $\hat G_n=[0,\hat\theta_n]$, where $\hat \theta_n$ is a random variable. Then, $|\hat G_n\triangle G|=|\hat\theta_n-\theta|$, and the performance of the estimator $\hat G_n$ of $G$ is that of the estimator $\hat \theta_n$ of the change-point $\theta$.
Let us define a least square estimator (LSE) of $\theta$. For $M=1,\ldots,n$, let 
	\begin{align}
		F(M) & =A(\{1,\ldots,M\})  \label{CPfunctionCrit} \\
		& = \sum_{i=1}^M(2Y_i-1). \nonumber
	\end{align}
Let $\hat M_n\in\underset{M=1,\ldots,n}{\operatorname{ArgMax}} \text{ } F(M)$, and $\hat\theta_n=X_{\hat M_n}$. We have the following theorem.

	\begin{theorem}\label{ThmChangePointUB}
		Let $n\geq 1$. Let Model \eqref{Model} hold, with design (DD). Let $\hat G_n=[0,\hat \theta_n]$. Then, 
		\begin{equation*}
			\sup_{G\in\mathcal S_0}\PP_G\left[|\hat G_n\triangle G|\geq \frac{x}{n}\right]\leq C_0e^{-x/(8\sigma^2)}, \forall x>0,
		\end{equation*}
		where $C_0$ is a positive constant which depends on $\sigma$ only.
	\end{theorem}

A simple application of Fubini's theorem yields the following result.

\begin{corollary}
	Let the assumptions of Theorem \ref{ThmChangePointUB} be satisfied. Then, for all $q>0$, there exists a positive constant $A_q$ which depends on $q$ and $\sigma$ only, such that
	\begin{equation*}
		\sup_{G\in\mathcal S_0} \E_G\left[|\hat G_n\triangle G|^q\right] \leq \frac{A_q}{n^q}.
	\end{equation*}
\end{corollary}

This corollary shows that the minimax risk on the class $\mathcal S_0$ is bounded from above by $1/n$, up to multiplicative constants.

\paragraph{Under the random design (RD)}

	We now consider Model \eqref{Model} hold, with design (RD). Assume that $G$ belongs to $\mathcal S_0$ and consider again the function $F(M)$ as defined in \eqref{CPfunctionCrit}, with $\hat M_n\in\underset{M=1,\ldots,n}{\operatorname{ArgMax}} \text{ } F(M)$, $\hat\theta_n=X_{\hat M_n}$ and $\hat G_n=[0,\hat \theta_n]$, as in the previous section. Then, the following theorem holds.

\begin{theorem}\label{ThmCPRD}
	Let $n\geq 1$ and let Model \eqref{Model} hold, with random design (RD). Then, the estimator $\hat G_n$ satisfies the following moment inequalities:
	\begin{equation*}
		\sup_{G\in\mathcal S_0} \E_G\left[|\hat G_n\triangle G|^q\right] \leq \frac{2C(2q)!(16\sigma^2)^{q}}{n^q},
	\end{equation*}
	for all positive integer $q$, where $C=2+16\sigma^2$.
Moreover, there exist universal constants $A_1$, $A_2$ and $A_3$ such that the following deviation inequality holds for all $x\geq 0$:

\begin{equation*}
	\PP\left[|\hat G_n\triangle G|\geq \frac{\sigma^2}{n}(x+A_1\sigma^2)\right]\leq e^{-\frac{x}{A_2\sqrt x+A_3\sigma^2}}.
\end{equation*}

\end{theorem}

\vspace{8mm}

Note that the estimator that we define is the same as in \cite[Section 1.9]{KTlectureNotes1993}, i.e., the least square estimator. However, we do not make use of a separation assumption of $\theta$ from $0$ and $1$, unlike in their proof, which shows that this assumption, which is only technical, is not linked to their estimation method. Next theorem shows that up to constants, $1/n$ is also a lower bound on the minimax risk for both designs (DD) and (RD). 

\begin{theorem}\label{ThmChangePointLB}
	Consider Model \eqref{Model}, with design (DD) or (RD). Then for all integer $n\geq 1$,
	\begin{equation*}
		\mathcal R_n(\mathcal S_0)\geq \frac{1}{8n}.
	\end{equation*}
\end{theorem}

Note that in the case of the deterministic design, the lower bound is determined by the fact that the parameter $\theta$ in \eqref{ChangePoint} is not identified by the model.

As a consequence, the minimax rate of convergence on the class $\mathcal S_0$ is $1/n$ under both designs (DD) and (RD):

\begin{theorem}\label{minimaxCP}
	Consider Model \eqref{Model}, with either design (DD) or (RD). Then, the minimax risk on the class $\mathcal S_0$ satisfies 
	\begin{equation*}
		\mathcal R_n(\mathcal S_0)\asymp_\sigma \frac{1}{n}.
	\end{equation*}
\end{theorem}

\subsection{Estimation of two change-points}
	Let us now assume that the unknown segment $G$ does not necessarily contain 0. We will prove that the classes $\mathcal S$ and $\mathcal S(\mu)$ yield two minimax rates that differ by a logarithmic factor.
	
	When $G$ is only assumed to belong to the largest class $\mathcal S$, we define the same LSE estimator as we did in \cite{Brunel2013} for convex polytopes and we prove an upper bound of the same order, i.e., $(\ln n)/n$. This estimator is obtained by maximising $A(G')$ (see \eqref{LSECriterion}) over all segments $G'$ with edge points that are integer multiples of $1/n$.
	
	If $|G|$ is \textit{a priori} known to be greater or equal to $\mu$, we split the available sample into two parts. With the first half, we define the same LSE estimator as in the general case seen above. This estimator may not be minimax optimal on $\mathcal S(\mu)$, but it is close to $G$ with high probability as shown in Theorem \ref{LSETheorem1}. Thus, the middle point $\hat m_n$ of this estimator is inside $G$ with high probability. On that event, we use the second half of the sample in order to estimate the two edge points of $G$, one on each side of $\hat m_n$, using the same technique as on the class $S_0$, where the base point $0$ is now replaced with $\hat m_n$.

\subsubsection{On the class $\mathcal S$}

	Let us first state the following theorem, which is, for the design (RD), a particular case of \cite[Theorem 1]{Brunel2013}, for $d=1$. 
	\begin{theorem}\label{LSETheorem1}
		Let $n\geq 2$. Let Model \eqref{Model} hold, with design (DD) or (RD).
		Let $\hat G_n\in\underset{G'\in\mathcal S}{\operatorname{ArgMax}} \text{ } A(G')$ be a LSE estimator of $G$. Then, there exist two positive constants $C_1$ and $C_2$ which depend on $\sigma$ only, such that
		\begin{equation*}
			\sup_{G\in\mathcal S} \PP_G\left[n\left(|\hat G_n\triangle G|-\frac{4\ln n}{C_2 n}\right)\geq x\right]\leq C_1e^{-C_2 x}, \forall x>0.
		\end{equation*}
	\end{theorem}
	
The expressions of $C_1$ and $C_2$ are given in the proof of \cite[Theorem 1]{Brunel2013}, for the design (RD). For the design (DD), we do not give a proof of this theorem here, but it can be easily adapted from that of the case of the design (RD).
The next corollary comes as an immediate consequence. 

\begin{corollary}
	Let the assumptions of Theorem \ref{LSETheorem1} be satisfied. Then, for all $q>0$, there exists a positive constant $B_q$ which depends on $q$ and $\sigma$ only, such that
	\begin{equation*}
		\sup_{G\in\mathcal S} \E_G\left[|\hat G_n\triangle G|^q\right] \leq B_q\left(\frac{\ln n}{n}\right)^q.
	\end{equation*}
\end{corollary}

This corollary shows that the minimax risk on the class $\mathcal S$ is bounded from above by $\ln(n)/n$, up to a multiplicative constant. The following theorem establishes a lower bound, if the noise is supposed to be Gaussian.

\begin{theorem}\label{thmLBBru13}
	Consider Model \eqref{Model}, with design (DD) or (RD). Assume that the noise terms $\xi_i$ are i.i.d. Gaussian random variables, with variance $\sigma^2>0$. For any large enough $n$, 
	\begin{equation*}
		\mathcal R_n(\mathcal S)\geq \frac{\alpha^2\sigma^2\ln n}{n},
	\end{equation*}
	where $\alpha$ is a universal positive constant.
\end{theorem}

This lower bound comes from \cite[Theorem 2]{Brunel2013} in the case of the design (RD), and the proof is easily adapted for the design (DD). As a consequence, the minimax risk on the class $\mathcal S$ is of the order $\ln(n)/n$:

\begin{theorem}\label{minimaxriskgen}
		Consider Model \eqref{Model}, with design (DD) or (RD). Assume that the noise terms $\xi_i$ are i.i.d. Gaussian random variables, with variance $\sigma^2>0$. The minimax risk on the class $\mathcal S$ satisfies, asymptotically:
		$$\mathcal R_n(\mathcal S) \asymp_\sigma \frac{\ln n}{n}.$$
\end{theorem}

\subsubsection{On the class $\mathcal S(\mu)$}

In this section we shall combine both Theorems \ref{ThmChangePointUB} and \ref{LSETheorem1} to find the minimax rate on the class $\mathcal S(\mu)$. 
Let Model \eqref{Model} hold and let $G\in\mathcal S(\mu)$. First, we split the sample into two equal parts. 
Let $\mathcal D_0$ be the set of sample points with even indices, and $\mathcal D_1$ the set of sample points with odd indices. Note that $\mathcal D_0\cup\mathcal D_1$ is exactly the initial sample, that these two subsamples are independent, and that each of them has at least $(n-1)/2$ ordered design points. Let $\hat G_n$ be the LSE estimator of $G$ given in Theorem \ref{LSETheorem1}, computed from the subsample $\mathcal D_0$, and let $\hat m_n$ be the middle point of $\hat G_n$. As it will be shown in the proof of the next theorem, $\hat m_n$ satisfies both following properties, with high probability:
\begin{enumerate}
	\item $\hat m_n\in G$,
	\item $\mu/4\leq\hat m_n\leq 1-\mu/4$.
\end{enumerate}

Then, each endpoint of $G$ is estimated separately, using the same technique as in Section \ref{Sec:OneChangePoint}.

Let $I_1$ be the set of odd integers between $1$ and $n$, so $\mathcal D_1=\{(X_i,Y_i):i\in I_1\}$. Define $I_1^+=\{i\in I_1:X_i\geq\hat m_n\}$ and $I_1^-=\{i\in I_1:X_i<\hat m_n\}$. Then, for all $i\in I_1^+$, $Y_i=\mathds 1(X_i\leq b)+\xi_i$ and for all $i\in I_1^-$, $Y_i=\mathds 1(X_i\geq a)+\xi_i=\mathds 1(1-X_i\leq 1-a)+\xi_i$, where $a$ and $b$ are the left and right boundaries of $G$, respectively. Let $\DS F_+(M)=\sum_{\substack{i\in I_1^+ \\ i\leq M}}(2Y_i-1)$ and $\DS F_-(M)=\sum_{\substack{i\in I_1^- \\ i\geq M}}(2Y_i-1)$. If $I_1^+=\emptyset$, let $\tilde M^+=n$, otherwise, let $\tilde M^+\in \underset{M\in I_1^+}{\operatorname{ArgMax}} \text{ } F_+(M)$. Similarly, if $I_1^+=\emptyset$, let $\tilde M^-=1$, otherwise, let $\hat M^-\in \underset{M\in I_1^-}{\operatorname{ArgMax}} \text{ } F_-(M)$. 
Finally, set $\tilde G_n=[\tilde a,\tilde b]$, where $\tilde a=X_{\hat M^-}$ and $\tilde b=X_{\hat M^+}$.

\paragraph{Under the design (DD)}

By adapting Theorem \ref{ThmChangePointUB}, we know that under the design (DD), both edge points of $G$ can be estimated at the speed $1/n$, up to multiplicative constants, when the two properties above are satisfied by $\hat m_n$. We get the following theorem.

\begin{theorem}\label{theoremMu}
		Consider Model \eqref{Model}, with design (DD). There exists an estimator $\tilde G_n$ of $G$, such that
				\begin{equation*}
					\sup_{G\in\mathcal S(\mu)}\PP_G\left[|\tilde G_n\triangle G|\geq \frac{x}{n}\right] \leq 2C_0e^{-\mu x/(256\sigma^2)}+C_1n^4e^{-C_2\mu n/2}, \forall x>0,
				\end{equation*}
		for $n$ large enough. The positive constants $C_0$ and $C_2$ are the same as in Theorems \ref{ThmChangePointUB} and \ref{LSETheorem1}.
\end{theorem}

Naturally, Theorem \ref{theoremMu} yields next corollary.

\begin{corollary}\label{corollCP}
	Let the assumptions of Theorem \ref{theoremMu} be satisfied. Then, for all $q>0$, there exists a positive constant $B'_q$ which depends on $q$, $\mu$ and $\sigma$ only, such that
	\begin{equation*}
		\sup_{G\in\mathcal S(\mu)}\E_G\left[|\tilde G_n\triangle G|^q\right] \leq \frac{B'_q}{n^q}.
	\end{equation*}
\end{corollary}

This corollary, for $q=1$, shows that the minimax risk on the class $\mathcal S(\mu)$ is bounded from above by $1/n$, up to a multiplicative constant.

 \paragraph{Under the design (RD)}
 
 When the random design (RD) is considered, we prove the following inequality, which is less strong than the deviation inequality obtained for the deterministic design, and yet enough for our purposes.

 \begin{theorem}\label{theoremMuRD}
		Let $n\geq 1$ and consider Model \eqref{Model}, with design (RD). There exist positive constants $c_j$, $j=1,\ldots,5$, that depend on $\sigma^2$ and $\mu$ only, such that the following holds.
		
\begin{equation*}
	\sup_{G\in\mathcal S(\mu)}\PP_G\left[|\tilde G_n\triangle G|\geq \frac{x+c_1}{n}\right] \leq c_2n^4e^{-c_3n}+c_4e^{-c_5\sqrt x},
\end{equation*}
for all $x\geq 0$.
\end{theorem}

As a consequence, we have the following moment inequalities.
 
 \begin{corollary}\label{corollCPRD}
	Let the assumptions of Theorem \ref{theoremMuRD} hold. Then, for all $q>0$, there exists a positive constant $B'_q$ which depends on $q$, $\mu$ and $\sigma$ only, such that
	\begin{equation*}
		\sup_{G\in\mathcal S(\mu)}\E_G\left[|\tilde G_n\triangle G|^q\right] \leq \frac{B'_q}{n^q}.
	\end{equation*}
\end{corollary}

As in the case of the deterministic design, this corollary shows that the minimax risk on the class $\mathcal S(\mu)$ is bounded from above by $1/n$, up to a multiplicative constant.

 \vspace{8mm}
         A very similar proof to that of Theorem \ref{ThmChangePointLB} yields a lower bound for this minimax risk, and we get the next theorem.

\begin{theorem}\label{minimaxlarge}
		Consider Model \eqref{Model}, with design (DD) or (RD). The minimax risk on the class $\mathcal S(\mu)$ satisfies, asymptotically,
		\begin{equation*}
			\mathcal R_n(\mathcal S(\mu))\asymp_{\mu,\sigma} \frac{1}{n}.
		\end{equation*}
\end{theorem}

\begin{remark}
	Note that, in Theorem \ref{theoremMu}, the upper bound contains one residual term which does not depend on $x$. This term, in order to be sufficiently small and to still yield Corollary \ref{corollCP} for $q=1$, requires that $\mu$ - if allowed to depend on $n$ - is of order larger than $(\ln n)/n$. This is consistent with Theorem \ref{TestThm} which shows that $(\ln n)/n$ is the order of magnitude of the length of the shortest segment that can be detected consistently, in a minimax sense. If $\mu$ becomes too small, i.e., of order less or equal to $\ln(n)/n$, then the lower bound proved for the minimax risk on $\mathcal S(\mu)$ will break down to $(\ln n)/n$ and the proof of the lower bound in Theorem \ref{thmLBBru13} can be applied to $\mathcal S(\mu)$, which would yield $\mathcal R_n(\mathcal S(\mu))\asymp_{\sigma} \mathcal R_n(\mathcal S)\asymp_{\sigma}\frac{\ln n}{n}$ as expected.
\end{remark}

\begin{remark}
We have only considered uniform or regular designs. However, the rates would be deteriorated if we allow densities that can get arbitrarily close to zero. For instance, as in \cite{Gaiffas2005}, if the density $f$ of the design points satisfies $f(x)\sim x^\beta$ as $x\to 0$, where $\beta>0$, then adapting the proofs yields lower bounds of the order of $n^{-1/(\beta+1)}$ for the minimax estimation rates on the three classes $\mathcal S_0$, $\mathcal S$ and $\mathcal S(\mu)$, which is strictly slower than $(\ln n)/n$.
\end{remark}

\section{Conclusion and discussion}\label{ConclusionDiscussion}

We summarize our results in Table \ref{Table1}. The rates that are written in this table hold when the noise is Gaussian. \\

\begin{table}[h]
\centering
\begin{tabular}{|c|c|c|c|}
\hline
& $\mathcal S$ & $\mathcal S_0$ & $\mathcal S(\mu)$ \\
\hline
Minimax rate for estimation & $\ln(n)/n$ & $1/n$ & $1/n$ \\
\hline
Separation rate for detection & $\ln(n)/n$ & $1/n$ & $\cdot$ \\
\hline
\end{tabular}
\caption{Minimax risks and separation rates for the classes $\mathcal S$, $\mathcal S_0$ and $\mathcal S(\mu)$.}
\label{Table1}
\end{table}

We have shown that asymptotically, a segment can be estimated infinitely faster when it is \textit{a priori} supposed either to contain a given point (here, $0$), or to be large enough. In terms of the change-point problem, the change-points can be detected and their location estimated at the fast rate $1/n$ either when one is known (i.e., there is only one unknown change-point), or when they are separated from one another. In particular, in the case of a single change-point, we show that the assumption that it is separated away from the boundaries of the domain, which is always imposed in change-point estimation (e.g., \cite{Ibragimov1975,Ibragimov1984,Korostelev1986,KTlectureNotes1993}) is only technical, not fundamental. When there are two change points that are not separated away from one another, we showed the presence of an extra logarithmic factor in both the minimax rate for estimation and separation rate for detection. As we mentioned in the introduction, \cite{ChanWalther2013} prove a similar result for detection in the case of two change-points that are not assumed to be separated from each other. However, in their work, the amplitude of the jump is not known. In our results, we have proved that even if the amplitude is known (it is 1 in our case), the logarithmic factor is unavoidable. 
%
%
%

Detection is a simpler problem than estimation. Hence, studying detection gives a benchmark for the optimal rates in estimation. Here, we did inference on segments, which are parametric objects, and we showed that detection and estimation are asymptotically equivalent, in the sense that in each scenario that we studied, separation rates for detection and minimax rates for estimation coincide. Because of the sharp jumps of the indicator function $\mathds 1(x\in G)$ at the boundaries of $G$, one would expect the parametric object $G$ to be estimable at the rate $n^{-1}$. However, we showed that in general, if $G$ can be arbitrarily small, hence, hard to detect, the rate $n^{-1}$ becomes negligible compared with the detection separated rate $(\ln n)/n$, which becomes the optimal estimation rate. If the detection threshold is at most of order $n^{-1}$, as it is the case in the class $\mathcal S_0$ or, obviously, in the class $\mathcal S(\mu)$, then $n^{-1}$ is the optimal estimation rate, as expected. 

In \cite{Brunel2013}, a higher dimensional version of the model is treated, where $G$ is a $d$-dimensional convex body ($d\geq 1$). In that case, $G$ is no longer a parametric object, and the optimal rate of convergence is affected by the complexity of the class of convex sets. Morally, the optimal rates of convergence are subsequent of a competition between separation rates for detectability and a terms that relates to the complexity of the class of sets $G$. In dimension 1, the complexity of the parametric class of segments yields a term of order $1/n$, which is never larger than the separation rate for detectability. In higher dimensions, \cite{Brunel2013} shows that the minimax rate of convergence for convex bodies is $n^{-2/(d+1)}$, although for the parametric class of convex polytopes with bounded number of vertices, it is exactly $(\ln n)/n$, which can be shown to be the separation rate. For the class of convex bodies, the term $n^{-2/(d+1)}$ dominates the separation rate, whereas for parametric classes of polytopes, the separation rate $(\ln n)/n$ dominates $1/n$ that is yielded by the complexity of the class.

In \cite{Brunel2013}, there are no assumptions of the type of Assumptions \ref{A1} or \ref{A2} and $G$ is estimated without any restriction on its size or location. We believe that if one restricts $G$ to having volume at least $\mu$, for some $\mu>0$ (corresponding to the scenario \ref{A2} here), there would not be any term of order $(\ln n)/n$ in the minimax rates, and for parametric classes, the minimax rates for estimation of $G$ would become $1/n$, i.e., the logarithmic factor would disappear. In a similar fashion, we believe that under a similar scenario as \ref{A1}, the separation rate would decrease to $1/n$ and the minimax rate for estimation would become $1/n$ as well for parametric classes (such as polytopes with bounded number of vertices). In higher dimensions, scenario \ref{A1} would have to be adapted. Indeed, one can easily modify the proof of the lower bound in \cite[Theorem 2]{Brunel2013} to show that the minimax rate is still at least $(\ln n)/n$, even under the extra assumption that $0\in G$. In fact, Assumption \ref{A1} should take the following form in dimension $d\geq 2$:
\begin{condition}\label{AA1}
	$[0,1]^{d-1}\times\{0\}\subseteq G.$
\end{condition}

Particular sets satisfying this assumption are boundary fragment (see \cite{KorostelevTsybakov1992} and \cite[Chap 3]{KTlectureNotes1993}). A boundary fragment in $\R^d$ ($d\geq 2$) is a set of the form 
$$G=\{(x,y)\in [0,1]^{d-1}\times \R: 0\leq y\leq g(x)\},$$
for some nonnegative function $g:[0,1]^{d-1}\to\R$, called the \textit{edge function} of $G$. As a side remark, in \cite{KorostelevTsybakov1992} and \cite[Chap 3]{KTlectureNotes1993}, the edge function $g$ is assumed to satisfy $h\leq g(x)\leq 1-h, \forall x\in [0,1]^{d-1}$, for some $h>0$ and we believe that, similarly to what we have proven here in dimension one, this assumption is only technical. A particular case of polytopes with bounded number of faces that satisfy Assumption \ref{AA1} is boundary fragments with edge function of the form $\displaystyle{g=\min_{1\leq k\leq K}f_k}$, where the $f_k$'s are affine maps on $[0,1]^{d-1}$ and $K$ is some fixed positive integer. We believe that the correct minimax rate of convergence on such a parametric class is $1/n$, although the only known upper bound at the moment is of order $(\ln n)/n$, see \cite{Brunel2013}.

\section{Proofs} \label{Sec:Proofs}

\subsection{Proof of Theorem \ref{TestThm}}

\paragraph{On the class $\mathcal S_0$}

\subparagraph{Upper bound}
	Let us first prove the upper bound, i.e. assume that $nh\rightarrow\infty$, and prove that there exists a consistent test. Recall that $N=\max\{i=1,\ldots,n : X_i\leq h\}=\#\left(\mathcal X\cap [0,h]\right)$. If the design is (DD), then $N$ is just equal to the integer part of $nh$. If the design (RD), then $N$ is a binomial random variable, with parameters $n$ and $h$.
Let us show first that the error of the first kind of the test $T_n^0$ goes to zero, when $n\rightarrow\infty$.
\begin{align*}
	\PP_\emptyset\left[S\leq cN\right] & = \PP_\emptyset\left[\#\left\{i=1,\ldots,N :Y_i>\frac{1}{2}\right\}\geq(1-c)N\right] \\
	& \leq \E\left[\PP_\emptyset\left[\#\left\{i=1,\ldots,N : \xi_i>\frac{1}{2}\right\}\geq(1-c)N|\mathcal X\right]\right].
\end{align*}
Since the $\xi_i$'s are independent of $\mathcal X$, the distribution of $\#\{i=1,\ldots,N : \xi_i>\frac{1}{2}\}$ conditionally to $\mathcal X$ is binomial, with parameters $N$ and $\beta$, where $\beta =\PP\left[\xi_1>1/2\right]\in [0,1)$.
Thus, by Bernstein's inequality for binomial random variables, by defining $\displaystyle{\gamma=\frac{(1-c-\beta)^2}{2\beta(1-\beta)+(1-c-\beta)/3}>0}$,
\begin{equation*}
	\PP_\emptyset\left[S\leq cN\right] \leq \E\left[\exp\left(-\gamma N\right)\right].
\end{equation*}
If $\mathcal X$ satisfies (DD), then $N\geq nh-1$ and it is clear that $\PP_\emptyset\left[S\leq cN\right]\longrightarrow 0$.
If $\mathcal X$ satisfies (RD), then 
\begin{equation*}
	\E\left[\exp\left(-\gamma N\right)\right]=\exp\left(-nh\left(1-e^{-\gamma}\right)\right),
\end{equation*} 
so $\PP_\emptyset\left[S\leq cN\right]\longrightarrow 0$.

Let us show, now, that the error of the second kind goes to zero as well. Let $G\in\mathcal S_0$ satisfying the alternative hypothesis, i.e. $|G|\geq h$. Denote by $\beta'=\PP[\xi_1\leq -1/2]$, and by $\displaystyle{\gamma'=\frac{(c-\beta')^2}{2\beta'(1-\beta')+(c-\beta')/3}>0}$
\begin{align*}
	\PP_G\left[S>cN\right] & = \PP_G\left[\#\left\{i=1,\ldots,N : Y_i\leq\frac{1}{2}\right\}> c N\right] \\
	& \leq \E\left[\PP_\emptyset\left[\#\left\{i=1,\ldots,N : \xi_i\leq-\frac{1}{2}\right\}>cN|\mathcal X\right]\right] \\
	& \leq \E\left[\exp\left(-\gamma' N\right)\right],
\end{align*}
by a similar computation to that for the error of the first kind. Since the right-side of the last inequality does not depend on $G$, 
\begin{equation*}
	\sup_{|G|\geq h}\PP_G\left[S>cN\right] \leq \E\left[\exp\left(-\gamma' N\right)\right]
\end{equation*}
and therefore, by the same argument as for the error of the first kind, goes to zero when $n\rightarrow\infty$, for both designs (DD) and (RD).

\subparagraph{Lower bound}

Assume, now, that $nh\rightarrow 0$. Let $\tau_n$ be any test. Let $G_1=[0,h]$. We denote by $\mathcal H$ the Hellinger distance between probability measures. The following computation uses properties of this distance, which can be found in \cite{Tsybakov2009}. 
\begin{align}
	\gamma_n(\tau_n,\mathcal C) & \geq \E_\emptyset\left[\tau_n\right]+\E_{G_1}\left[1-\tau_n\right] \nonumber \\
	& \geq \int \min\left(d\PP_\emptyset,d\PP_{G_1}\right) \nonumber \\
	& \geq \frac{1}{2}\left(1-\frac{\mathcal H(\PP_\emptyset,\PP_{G_1})^2}{2}\right)^2. \label{LBTest1}
\end{align}
Let $G,G'\in\mathcal S$. A simple computation shows that, for the design (DD),
\begin{equation}\label{hellinger1}
	1-\frac{\mathcal H(\PP_G,\PP_{G'})^2}{2} = \exp\left(-\frac{\#\left(\mathcal X\cap (G\triangle G')\right)}{8\sigma^2}\right),
\end{equation}
and for the design (RD), 
\begin{equation}\label{hellinger2}
	1-\frac{\mathcal H(\PP_G,\PP_{G'})^2}{2} = \left(1-\left(1-e^{-\frac{1}{8\sigma^2}}\right)|G\triangle G'|\right)^n.
\end{equation}
In particular, for the design (DD), 
\begin{equation*}
	1-\frac{\mathcal H(\PP_\emptyset,\PP_{G_1})^2}{2} \geq \exp\left(-\frac{nh}{8\sigma^2}\right),
\end{equation*}
and for the design (RD),
\begin{equation*}
	1-\frac{\mathcal H(\PP_\emptyset,\PP_{G_1})^2}{2} = \left(1-\left(1-e^{-\frac{1}{8\sigma^2}}\right)h\right)^n.
\end{equation*}
In both cases, we showed that the right side of \eqref{LBTest1} tends to $1/2$, when $n\rightarrow\infty$. Therefore the test $\tau_n$ is not consistent.

\paragraph{On the class $\mathcal S$}

\subparagraph{Upper bound}

Assume that $\displaystyle{\frac{nh}{\ln n}\longrightarrow\infty}$. Let us first show that the error of the first kind of $T_n^1$ goes to zero, when $n\rightarrow\infty$. Recall that $T_n^1=\mathds 1(R\geq 0)$, where $R=\sup_{|G|\geq h}R(G)$ and $R(G)=\sum_{i=1}^n Y_i\mathds 1(X_i\in G)-\frac{\#(\mathcal X\cap G)}{2}$, for all $G\in\mathcal S$. Note that $R(G)$ is piecewise constant, and can only take a finite number of values. It is clear that 
\begin{equation*}
	\{R(G):G\in\mathcal S,|G|\geq h\}=\{R([X_k,X_l)):1\leq k< l\leq n, X_l-X_k>h\}.
\end{equation*}
Recall that for $1\leq k< l\leq n$, $\displaystyle{R([X_k,X_l))=\frac{1}{2}\sum_{i=k}^{l-1}(2Y_i-1)}$.
Therefore, 
\begin{align}
	\PP_\emptyset[R\geq 0] & = \PP_\emptyset\left[\max_{\substack{1\leq k < l\leq n \\ X_l-X_k>h}} R([X_k,X_l))>0\right] \nonumber \\
	& \leq \PP_\emptyset\left[\bigcup_{1\leq k < l\leq n} \left\{R([X_k,X_l))>0\right\}\cap\left\{X_l-X_k>h\right\}\right] \nonumber \\
	& \leq \sum_{1\leq k < l\leq n}\PP_\emptyset\left[R([X_k,X_l))>0,X_l-X_k>h\right] \nonumber \\
	& \leq \sum_{1\leq k < l\leq n}\PP_\emptyset\left[\sum_{i=k}^{l-1}(2\xi_i-1)>0\right]\PP[X_l-X_k>h]. \label{1stErrS1}
\end{align}
For $1\leq k<l\leq n$, 
\begin{equation}
	\PP_\emptyset\left[\sum_{i=k}^{l-1}(2\xi_i-1)>0\right]\leq \exp\left(-\frac{(l-k)\sigma^2}{8}\right), \label{1stErrS2}
\end{equation}
using Markov's inequality and \eqref{subgauss}.

If the design is (DD), then $\PP[X_l-X_k>h]$ is 1 if and only if $l-k>nh$, 0 otherwise, so from \eqref{1stErrS1} and \eqref{1stErrS2} we get that
\begin{align*}
	\PP_\emptyset[R\geq 0] & \leq \sum_{l-k >nh}\exp\left(-\frac{(l-k)\sigma^2}{8}\right) \\
	& \leq \sum_{l-k >nh}\exp\left(\frac{(-nh)\sigma^2}{8}\right) \\
	& \leq \frac{n^2}{2}\exp\left(\frac{(-nh)\sigma^2}{8}\right) \longrightarrow 0,
\end{align*}
when $n\rightarrow\infty$, which proves that the error of the first kind goes to zero.

If the design is (RD), let us use the following Lemma.

\begin{lemma}\label{lemmaRD}
	Let $X_1,\ldots,X_n$ be the (RD) design. Then, for any $1\leq k<l\leq n$, and $h>0$,
	\begin{equation*}
		\PP[X_l-X_k>h] \leq  n\exp\left(-(n-1)h(1-e^{-u})+u(l-k)\right), \forall u>0.
	\end{equation*}
\end{lemma}

\paragraph{Proof of Lemma \ref{lemmaRD}}
Note that the event $\{X_l-X_k>h\}$ is equivalent to $\{\#(\mathcal X\cap (X_k,X_k+h)) < l-k\}$. Let us denote by $X_1',\ldots,X_n'$ the preliminary design, from which $X_1,\ldots,X_n$ is the reordered version. The random variables $X_1',\ldots,X_n'$ are then i.i.d., with uniform distribution on $[0,1]$. Hence,
\begin{align}
	\PP[X_l-X_k>h] & = \sum_{j=1}^n \PP\left[\#(\mathcal X\cap (X_k,X_k+h)) < l-k, X_k=X_j'\right] \nonumber \\
	& \leq \sum_{j=1}^n \PP\left[\#(\mathcal X\cap (X_j',X_j'+h)) < l-k\right] \nonumber \\
	& = \sum_{j=1}^n \E\left[\PP\left[\#(\mathcal X\cap (X_j',X_j'+h)) < l-k|X_j'\right]\right] \nonumber \\
	& = \sum_{j=1}^n \E\left[f(X_j')\right], \label{1stErrS3}
\end{align}
where $f(x)=\PP\left[\#(\tilde {\mathcal X}\cap (x,x+h)) < l-k\right]$, for all $x\in[0,1]$, and $\tilde {\mathcal X}=\{X_1',\ldots,X_{n-1}'\}$.
The random variable $\#(\tilde{\mathcal X}\cap (x,x+h))$ is binomial with parameters $n-1$ and $h$, hence, $f(x)=\PP[N\geq n-1-(l-k)]$, where $N$ is a binomial random variable with parameters $n-1$ and $1-h$. By Markov's inequality, for all $u>0$,
\begin{align}
	\label{1stErrS4} f(x) & \leq e^{-u(n-1-l+k)}\E[e^{uN}] \nonumber \\
	& = e^{-u(n-1-l+k)}\left((1-h)e^u+h\right)^{n-1} \nonumber \\
	& = e^{u(l-k)}\left((1-h(1-e^{-u})\right)^{n-1} \nonumber \\
	& \leq e^{u(l-k)}e^{-(n-1)h(1-e^{-u})},
\end{align}
where we used the inequality $1-x\leq e^{-x}, \forall x\in\R$ in the last line. Then, \eqref{1stErrS3} and \eqref{1stErrS4} yield the lemma. \hfill $\Box$\\[0.4cm]

Therefore, by \eqref{1stErrS1}, \eqref{1stErrS2} and \eqref{1stErrS3}, and Lemma \ref{lemmaRD} with $u=\sigma^2/8$,
\begin{align*}
	\PP_\emptyset[R\geq 0] & \leq \sum_{1\leq k < l\leq n}n\exp\left(-\frac{(l-k)\sigma^2}{8}-(n-1)h(1-e^{-u})+u(l-k)\right) \\
	& \leq \frac{n^3}{2}\exp\left(-(n-1)h(1-e^{-\sigma^2/8})\right) \longrightarrow 0,
\end{align*}
when $n\rightarrow \infty$, which proves that the error of the first kind goes to zero.

Let us bound, now, the error of the second kind. Let $G\in\mathcal S$ satisfying $|G|\geq h$. For this $G$, denote by $N_G=\#(\mathcal X\cap G)$. Then,
\begin{align}
	\PP_G[R<0] & \leq \PP_G\left[R(G)\leq N_G/2\right] \nonumber \\
	& \leq \PP\left[\sum_{i=1}^n \xi_i\mathds 1(X_i\in G)\leq -N_G/2\right]. \label{PrUB2_S}
\end{align}

For the design (DD), $N_G$ is the integer part of $n|G|$, so $N_G\geq nh$. Therefore, by Markov's inequality, and by \eqref{subgauss}, \eqref{PrUB2_S} becomes 
\begin{equation}
	\PP_G[R<0] \leq \exp\left(-\frac{\sigma^2N_G}{8}\right) \leq \exp\left(-\frac{nh\sigma^2}{8}\right). \label{DDUB2}
\end{equation}

For the design (RD), $N_G$ is a random binomial variable, with parameters $n$ and $|G|$. By conditioning to the design and using Markov's inequality, \eqref{PrUB2_S} becomes 
\begin{align}
	\PP_G[R<0] & \leq \PP\left[\sum_{i=1}^n -\xi_i\mathds 1(X_i\in G)\geq N_G/2\right] \nonumber \\
	& \leq \E\left[\exp\left(-\frac{\sigma^2 N_G}{8}\right)\right] \nonumber \\
  & \leq \exp\left(-Cn|G|\right) \leq \exp\left(-Cnh\right), \label{RDUB2}
\end{align}
where $C=1-e^{-\frac{\sigma^2}{8}}$.

In both cases \eqref{DDUB2} and \eqref{RDUB2}, the right side does not depend on $G$, and goes to zero as $n\rightarrow\infty$. We conclude that, for both designs (DD) and (RD), 
\begin{equation*}
	\sup_{|G|\geq h}\PP_G[R<0] \longrightarrow 0,
\end{equation*}
which ends the proof of the upper bound.

\subparagraph{Lower bound}

We more or less reproduce the proof of \cite{Gayraud2001}, Theorem 3.1. Here, the noise is supposed to be zero-mean Gaussian, with variance $\sigma^2$. Let us assume that $\displaystyle{\frac{nh}{\ln n}\longrightarrow 0}$. Let $M=1/h$, assumed to be an integer, without loss of generality. For $q=1,\ldots,M$, let $G_q=[(q-1)h,qh]$. 
For $q=1,\ldots,M$, let $Z_q=\frac{d\PP_{G_q}}{d\PP_\emptyset}(X_1,Y_1,\ldots,X_n,Y_n)$, and denote by $\bar Z=\frac{1}{M}\sum_{q=1}^M Z_q$. Let $\tau_n$ be any test. Then,
\begin{align}
	\gamma_n(\tau_n,\mathcal S) & \geq \PP_\emptyset\left[\tau_n=1\right]+\frac{1}{M}\sum_{q=1}^M\PP_{G_q}\left[\tau_n=0\right] \nonumber \\
	& \geq \frac{1}{M}\sum_{q=1}^M\left(\PP_\emptyset\left[\tau_n=1\right]+\PP_{G_q}\left[\tau_n=0\right]\right) \nonumber \\
	& \geq \frac{1}{M}\sum_{q=1}^M\left(\E_\emptyset\left[\tau_n\right]+\E_{G_q}\left[1-\tau_n\right]\right) \nonumber \\
	& \geq \frac{1}{M}\sum_{q=1}^M\E_\emptyset\left[\tau_n+(1-\tau_n)Z_q\right] \nonumber \\
	& \geq \E_\emptyset\left[\tau_n+(1-\tau_n)\bar Z\right] \nonumber \\
	& \geq \E_\emptyset\left[\left(\tau_n+(1-\tau_n)\bar Z\right)\mathds 1(\bar Z\geq 1/2)\right] \nonumber \\
	& \geq \frac{1}{2}\PP_\emptyset\left[\bar Z\geq 1/2\right]. \label{LBtest01}
\end{align}
Let us prove that $\E_\emptyset[\bar Z]=1$, and that $\VV_\emptyset[\bar Z]\longrightarrow 0$. This will imply that the right side term of \eqref{LBtest01} goes to $1/2$, when $n\rightarrow\infty$. 

For $q=1,\ldots,M$, under the null hypothesis, 
\begin{align}
	Z_q & = \exp\left(-\frac{1}{2\sigma^2}\sum_{i=1}^n\left((Y_i-\mathds 1(X_i\in G_q))^2-Y_i^2\right)\right) \nonumber \\
	& = \exp\left(\frac{1}{2\sigma^2}\sum_{i=1}^n(2\xi_i-1)\mathds 1(X_i\in G_q)\right). \label{Zq}
\end{align}
By its definition, $Z_q$ has expectation 1 under $\PP_\emptyset$:
\begin{equation}\label{expectZ}
	\E_\emptyset[\bar Z]=1.
\end{equation}

Since, almost surely, no design point falls in two $G_q$'s at the time, a simple computation shows that the random variables $Z_q, q=1,\ldots,M$, are not correlated. Thus,
$$\VV_\emptyset[\bar Z]=\frac{1}{M^2}\sum_{q=1}^M\VV_\emptyset[Z_q].$$
Let us bound from above $\VV_\emptyset[Z_q]$, for $q=1,\ldots,M$:
\begin{align}
	\VV_\emptyset[Z_q] & \leq \E_\emptyset[Z_q^2] \nonumber \\
	& = \E\left[\exp\left(-\frac{\#(\mathcal X\cap G_q)}{\sigma^2}\right)\E_\emptyset\left[\exp\left(\frac{2}{\sigma^2}\sum_{i=1}^n\xi_i\mathds 1(X_i\in G_q)\right) | \mathcal X\right]\right] \nonumber \\
	& = \E\left[\exp\left(\frac{\#(\mathcal X\cap G_q)}{\sigma^2}\right)\right]. \label{boundonvariance}
\end{align}

If the design is (DD), then we get that 
\begin{equation*}
	\VV_\emptyset[Z_q]\leq \exp\left(\frac{nh+1}{\sigma^2}\right),
\end{equation*}
and the variance of $\bar Z$ is then bounded from above:
\begin{equation}
	\VV_\emptyset[\bar Z]\leq h\exp\left(\frac{nh+1}{\sigma^2}\right). \label{var1}
\end{equation}

If the design is (RD), then $\#(\mathcal X\cap G_q)$ is a binomial random variable with parameters $n$ and $h$, so from \eqref{boundonvariance}, we get that

\begin{align*}
	\VV_\emptyset[Z_q] & \leq \left(1+\left(e^{1/\sigma^2}-1\right)h\right)^n \\
	& \leq \exp\left(Cnh\right),
\end{align*}
where $C=e^{1/\sigma^2}-1$, and the variance of $\bar Z$ is then bounded from above:
\begin{equation}
	\VV_\emptyset[\bar Z]\leq h\exp\left(Cnh\right). \label{var2}
\end{equation}

Since we assumed that $nh/\ln n\longrightarrow 0$, the right side terms of \eqref{var1} and \eqref{var2} go to zero, and therefore, for both designs (DD) and (RD),
\begin{equation}\label{boundvarianceZDD}
	\VV_\emptyset[\bar Z]\longrightarrow 0.
\end{equation} 

Finally, we get from \eqref{LBtest01}, \eqref{expectZ} and \eqref{boundvarianceZDD}, that 
\begin{equation*}
	\underset{n\rightarrow\infty}{\operatorname{liminf}}\gamma_n(\tau_n,\mathcal S)\geq \frac{1}{2}.
\end{equation*}
This concludes the proof. \hfill $\blacksquare$

\subsection{Proof of Theorem \ref{ThmChangePointUB}}

The beginning of this proof holds for any design $\{X_1,\ldots,X_n\}$, independent of the noise $\xi_i, i=1,\ldots,n$.
Let $G\in\mathcal S_0$. Let $M=\max\{i=1,\ldots,n : X_i\in G\}$ - set $M=0$ if the set is empty -. Then, $\hat M_n\in\underset{M'=1,\ldots,n}{\operatorname{ArgMax}} \text{ } \left(F(M')-F(M)\right)$, and, by \eqref{GausProc},

\begin{equation*}
	F(M')-F(M)= -|M'-M| + { \left\{
    \begin{array}{l}
        2\sum_{i=M'+1}^M\xi_i   \mbox{ }\mbox{ if } M>M',   \vspace{3mm} \\
				0 \mbox{ }\mbox{ if } M'=M, \vspace{3mm} \\
        -2\sum_{i=M+1}^{M'}\xi_i   \mbox{ }\mbox{ if } M<M'.
    \end{array}
	\right.}
\end{equation*}

Let us complete the i.i.d. sequence $\xi_1,\ldots,\xi_n$ to obtain an infinite double sided i.i.d. sequence $(\xi_i)_{i\in\mathbb Z}$, independent of the design.
Let $k\in\mathbb N^*$ be any positive integer. Define, for $i\in\mathbb Z, \tilde\xi_i=\xi_{i+M}$. Since $M$ depends on the design only, it is independent of the $\xi_i,i\in\mathbb Z$, and therefore, the $\tilde\xi_i,i\in\mathbb Z$ are i.i.d., with same distribution as $\xi_1$.
Let $E_k$ be the event $\{\hat M_n\geq M+k\}$. If $E_k$ holds, then $F(j)-F(M)\geq 0$ for some $j\geq M+k$, yielding 
\begin{align*}
	0 & \leq \max_{M+k\leq j\leq n}\left(M-j - 2\sum_{i=M+1}^{j}\xi_i\right) \\
	& = \max_{M+k\leq j\leq n}\left(M-j - 2\sum_{i=1}^{j-M}\tilde\xi_i\right) \\
	& = \max_{k\leq j\leq n-M}\left(-j - 2\sum_{i=1}^{j}\tilde\xi_i\right) \\
	& \leq \max_{j\geq k}\left(-j - 2\sum_{i=1}^{j}\tilde\xi_i\right).
\end{align*}
Hence, for all $u>0$,
\begin{align*}
	\PP_G[E_k] & \leq \PP_G\left[\max_{k\leq j}\left(-j - 2\sum_{i=1}^{j}\hat\xi_i\right) \geq 0\right] \\
	& \leq \PP\left[\max_{k\leq j}\left(-j - 2\sum_{i=1}^{j}\xi_i\right) \geq 0\right] \\
	& \leq \sum_{j=k}^\infty \PP\left[-2\sum_{i=1}^{j}\xi_i \geq j\right] \\
	& \leq \sum_{j=k}^\infty \E\left[e^{-2u\sum_{i=1}^{j}\xi_i}\right]e^{-uj}, \mbox{ by Markov's inequality} \\
	& \leq \sum_{j=k}^\infty e^{(-u+2\sigma^2u^2)j}, \mbox{ by \eqref{subgauss}}
\end{align*}
and, by choosing $u=1/(4\sigma^2)$,
\begin{equation*}
	\PP_G[E_k]\leq Ce^{-k/(8\sigma^2)},
\end{equation*}
where $C=\left(1-e^{-1/(8\sigma^2)}\right)^{-1}$ is a positive constant.
By symmetry, we obtain that :
\begin{equation}\label{devind}
	\PP_G[|\hat M_n-M|\geq k]\leq 2Ce^{-k/(8\sigma^2)}.
\end{equation}

If the design is (DD), the conclusion is straightforward, since for all $i,j=1,\ldots,n, |X_i-X_j|=\frac{|i-j|}{n}$, and Theorem \ref{ThmChangePointUB} is proved. \hfill $\blacksquare$ \\[0.2cm]

\subsection{Proof of Theorem \ref{ThmCPRD}}

First, note that inequality \eqref{devind} holds for any design. Hence, under the random design (RD),

\begin{equation}
	\PP_G[|\hat M_n-M|\geq k]\leq 2Ce^{-k/(8\sigma^2)}, \label{repeat}
\end{equation}

where we recall that $M=\max\{i=1,\ldots,n : X_i\in G\}$, $M=0$ if the set is empty. 

Let $q$ be a positive integer. We aim to bound $\E_G\left[|X_{\hat M_n}-\theta|^q\right]$ from above, where $G=[0,\theta]$, $\theta\in [0,1]$.

\begin{align}
	\E_G\left[|X_{\hat M_n}-\theta|^q\right] & = \sum_{k=-\infty}^\infty \E_G\left[|X_{\hat M_n}-X_M|^q \mathds 1(\hat M_n=M+k)\right] \nonumber \\
	& = \sum_{k=-\infty}^0 \E_G\left[|X_{\hat M_n}-\theta|^q \mathds 1(\hat M_n=M+k)\right] \nonumber \\
	& \hspace{8mm} + \sum_{k=1}^\infty \E_G\left[|X_{\hat M_n}-\theta|^q \mathds 1(\hat M_n=M+k)\right]. \label{stepNewProof1}
\end{align}
Let us bound from above the second sum only. The first one requires exactly the same techniques. 

Just for the sake of notation, we set $X_j$ to zero whenever $j$ is not in the range $\{1,\ldots,n\}$ and we identify $X_M$ with $\theta$ in the next computation.
\begin{align}
	\sum_{k=1}^\infty \E_G & \left[|X_{\hat M_n}-\theta|^q \mathds 1(\hat M_n=M+k)\right] \nonumber \\
	& = \sum_{k=1}^\infty \E_G\left[(X_{M+k}-\theta)^q \mathds 1(\hat M_n=M+k)\right] \nonumber \\
	& \leq \sum_{k=1}^\infty k^{q-1}\sum_{j=0}^{k-1} \E_G\left[(X_{M+j+1}-X_{M+j})^q \mathds 1(\hat M_n=M+k)\right] \nonumber \\
	& \leq \sum_{k=1}^\infty k^{q-1}\sum_{j=0}^{k-1} \E_G\left[(X_{M+j+1}-X_{M+j})^{2q}\right]^{1/2}\PP_G\left[\hat M_n=M+k\right]^{1/2}, \label{stepNewProof2}
\end{align}
where the last line follows from Cauchy-Schwarz inequality.
By \eqref{repeat}, for $k\geq 1$,
\begin{equation} \label{stepNewProof3}
	\PP_G\left[\hat M_n=M+k\right] \leq \PP_G\left[|\hat M_n-M|\geq k\right] \leq  2Ce^{-k/(8\sigma^2)}.
\end{equation}
Let $j$ be a nonnegative integer. Conditioning on the random variable $M$ yields
\begin{equation}\label{stepNewProof4}
	\E_G\left[(X_{M+j+1}-X_{M+j})^{2q}\right] = \sum_{l=1}^n\E_G\left[(X_{M+j+1}-X_{M+j})^{2q}|M=l\right]\PP[M=l]. 
\end{equation}
To compute the first term in the right hand side of \eqref{stepNewProof4}, we use the following lemma. For all vectors $(Y_1,\ldots,Y_n)$ with pairwise disjoint entries, we denote by $(Y_{(1)},\ldots,Y_{(n)})$ their increasing reordering, i.e., $Y_{(1)}<\ldots<Y_{(n)})$.

\begin{lemma}\label{Lemma2}
	Let $n$ be a positive integer, $\theta\in[0,1]$ and $Y_1,\ldots,Y_n$ be i.i.d. random variables with uniform distribution in $[0,1]$. Let $M=\max\{i=1,\ldots,n:Y_{(i)}\leq \theta\}$ and $l\in\{1,\ldots,n-1\}$. Let $Z_1,\ldots,Z_{n-l}$ be i.i.d random variables with uniform distribution in $[\theta,1]$.
Then, the conditional distribution of $(Y_{(M+1)},\ldots,Y_{(n)})$ conditional on the event $\{M=l\}$ is equal to the distribution of $Z_{(1)},\ldots,Z_{(n-l)}$.
\end{lemma}

\begin{proof}

Recall that the joint density of $(Y_{(1)},\ldots,Y_{(n)})$ is given by $$n! \mathds 1(0\leq y_1\leq\ldots\leq y_n),\hspace{8mm} y_1,\ldots,y_n\in\R.$$
Let $f:\R^{n-l}\longrightarrow\R$ be a continuous and bounded function. Then,
\begin{align}
	\E & \left[f(Y_{(M+1)},\ldots,Y_{(n)})|M=l\right] \nonumber \\
	& = \E\left[f(Y_{(l+1)},\ldots,Y_{(n)})\mathds 1(M=l)\right] \left(\PP[M=l]\right)^{-1} \nonumber \\
	& = \int_{[0,1]^n} f(y_{l+1},\ldots,y_n)\mathds 1(y_1\leq\ldots\leq y_l\leq\theta\leq y_{l+1}\leq\ldots\leq y_n)dy_1\ldots dy_n \nonumber \\
	& \hspace{12mm}\times n!\left(\PP[M=l]\right)^{-1} \nonumber \\
	& =  \int_{[0,1]^{n-l}} f(y_{l+1},\ldots,y_n) \mathds 1(\theta\leq y_{l+1}\leq\ldots\leq y_n)dy_{l+1}\ldots dy_n \nonumber \\
	& \hspace{12mm} \times n! \int_{[0,1]^l}\mathds 1(y_1\leq\ldots\leq y_l\leq\theta) \left(\PP[M=l]\right)^{-1},
\end{align}
which shows that the conditional joint density of $(Y_{(l+1)},\ldots,Y_{(n)})$ is proportional to $\mathds 1(\theta\leq y_{l+1}\leq\ldots\leq y_n\leq 1)$ and therefore shows Lemma \ref{Lemma2}. 

\end{proof}

As a consequence of this lemma, for all $l\in\{1,\ldots,n\}$ and $j\in\{0,\ldots,n-l-1\}$,
\begin{align*}
	\E_G\left[(X_{M+j+1}-X_{M+j})^{2q}|M=l\right] & = \E\left[\left(\min(Z_1,\ldots,Z_{n-l})-\theta\right)^{2q}\right] \\
	& = 2q(1-\theta)^{2q}\frac{\Gamma(2q)\Gamma(n-l+1)}{\Gamma(n+2q-l+1)},
\end{align*}
where $Z_1,\ldots,Z_{n-l}$ are i.i.d. unformly distributed in $[\theta,1]$ and $\Gamma$ is Euler's Gamma function.

If $j\geq n-l$, then the previous conditional expectation is zero. 

In addition,
$$\PP[M=l]={n\choose l}\theta^l(1-\theta)^{n-l}.$$
Hence, following \eqref{stepNewProof4}, for $j\in\{0,\ldots,k-1\}$
\begin{align}\label{stepNewProof5}
	\E_G & \left[(X_{M+j+1}-X_{M+j})^{2q}\right] \nonumber \\
	& = \sum_{l=1}^n\frac{n!}{l!(n-l)!}\theta^l(1-\theta)^{n-l}2q(1-\theta)^{2q}\frac{\Gamma(2q)\Gamma(n-l+1)}{\Gamma(n+2q-l+1)} \nonumber \\
	& = \frac{n!(2q)!}{(n+2q)!}\sum_{l=1}^n\frac{(n+2q)!}{l!(n+2q-l)!}\theta^l(1-\theta)^{n+2q-l} \nonumber \\
	& \leq \frac{n!(2q)!}{(n+2q)!} \nonumber  \\
	& \leq \frac{(2q)!}{n^{2q}}.
\end{align}

Finally, combining \eqref{stepNewProof2} with \eqref{repeat} and \eqref{stepNewProof5} yields:

\begin{equation} \label{ExtraStep1}
	\sum_{k=1}^\infty \E_G\left[|X_{\hat M_n}-\theta|^q \mathds 1(\hat M_n=M+k)\right] \leq \frac{\sqrt{(2q)!}}{n^q} \sum_{k=1}^\infty k^q e^{-k/(16\sigma^2)}.
\end{equation}	
	
In order to bound the sum in the right hand side of \eqref{ExtraStep1}, we use the following lemma.

\begin{lemma} \label{ExtraLemma}
	For all $\lambda>0$ and all $q>0$,
	$$\sum_{k=1}^\infty k^qe^{-\lambda k}\leq (2+\lambda^{-1})\frac{q!}{\lambda^{q+1}}.$$
\end{lemma}

\begin{proof}
	Denote by $\phi(x)=x^qe^{-\lambda x}$, for all $x\geq 0$. Then, $\phi$ is positive and it is increasing on $[0,q/\lambda]$ and decreasing on $[q/\lambda,\infty)$. Therefore, one can write:
	
\begin{align}\label{ExtraLemmaStep1}
	\sum_{k=1}^\infty \phi(k) & \leq \sum_{k\leq q/\lambda-1}\phi(k)+2\phi(q/\lambda)+\sum_{k\geq q/\lambda+1}\phi(k) \nonumber \\
	& \leq \int_0^{q/\lambda} \phi(t)\diff t+2\phi(q/\lambda)+\int_{q/\lambda}^\infty \phi(t)\diff t \nonumber \\
	& = \frac{q!}{\lambda^{q+1}}+2\phi(q/\lambda).
\end{align}
Note that 
\begin{equation*}
	\phi(q/\lambda)=q^qe^{-q}\lambda^{-q} \leq \frac{q!}{\lambda^q}.
\end{equation*}

Hence, \eqref{ExtraLemmaStep1} yields the conclusion of Lemma \ref{ExtraLemma}.

\end{proof}

Thus, for $\lambda=(16\sigma^2)^{-1}$, \eqref{ExtraStep1} continues as

\begin{align*}
	\sum_{k=1}^\infty \E_G\left[|X_{\hat M_n}-\theta|^q \mathds 1(\hat M_n=M+k)\right] & \leq C\frac{\sqrt{(2q)!q!}}{\lambda^q n^q} \\
	& \leq C\frac{(2q)!}{\lambda^q n^q},
\end{align*}
where we used $(q!)^2\leq (2q)!$ in the last inequality and $C=2+\lambda^{-1}$.

Similar arguments would yield the same upper bound for the first sum in \eqref{stepNewProof1}. Therefore, 
\begin{equation*}
\E_G\left[|X_{\hat M_n}-\theta|^q\right]\leq 2C\frac{(2q)!}{\lambda^q n^q}, \quad \forall q\geq 1.
\end{equation*}

Finally, we use the following lemma, that is similar to a Bernstein-type inequality, for non centered random variables.

\begin{lemma} \label{Bernstein}
Let $A\geq 1$ and $\alpha>0$ be given numbers. Let $Z$ be a random variable satisfying $\DS \E[Z^q]\leq A(2q)!\alpha^{2q}$ for all positive integer $q$.
Then, for all $x\geq 0$, 
$$\PP\left[Z\geq \alpha^2(2x+4A)\right]\leq e^{-\frac{x}{8\sqrt x+192A}}.$$
\end{lemma}

\begin{proof}
	Let $T=\sqrt Z$. Then, the second assumption yields 
\begin{equation} \label{AssMoment1}
	\E[T^{2q}]\leq A(2q)!\alpha^{2q}, \forall q\geq 1.
\end{equation}
As a consequence, for all integer $q\geq 2$, H{\"o}lder's inequality yields  
\begin{equation} \label{AssMoment2}
	\E[T^{2q-1}]\leq \E[T^{2q}]^{\frac{2q-1}{2q}}\leq A(2q)!\alpha^{2q-1},
\end{equation}
since $A\geq 1$. In particular, \eqref{AssMoment1} and \eqref{AssMoment2} yield 
\begin{equation}\label{AssMoment3}
	\E[T^q]\leq A(q+1)!\alpha^q, \quad \forall q\geq 1.
\end{equation}

Let $t\geq 0$. By Markov's inequality, for all $\kappa>0$, $\DS \PP[T-\E[T]\geq t]\leq e^{-\kappa t}\E\left[e^{\kappa(T-\E[T]}\right]$ and a Taylor expansion yields
\begin{align*}
	\PP[T-\E[T]\geq t]& \leq e^{-\kappa t}\left(1+\sum_{q=2}^\infty \frac{\kappa^q}{q!}\E\left[(T-\E[T])^q\right] \right) \\
	& \leq e^{-\kappa t}\left(1+\sum_{q=2}^\infty \frac{\kappa^q}{q!}2^{q-1}\E\left[T^q+\E[T]^q\right]\right) \\
	& \leq e^{-\kappa t}\left(1+\sum_{q=2}^\infty \frac{\kappa^q}{q!}2^q\E\left[T^q\right]\right),
\end{align*}
where we used the fact that $\E[T]^q\leq \E[T^q]$ in the last inequality. Hence, by \eqref{AssMoment3}, if $\kappa\leq 1/(4\alpha)$,

\begin{align} \label{StepBern1}
	\PP[T-\E[T]\geq t] & \leq e^{-\kappa t}\left(1+A\sum_{q=2}^\infty (q+1)(2\kappa\alpha)^q\right) \nonumber \\
	& \leq e^{-\kappa t}\left(1+4A(\kappa\alpha)^2\sum_{q=0}^\infty (q+3)2^{-q}\right) \nonumber \\
	& = e^{-\kappa t}\left(1+48A(\kappa\alpha)^2\right) \nonumber \\
	& \leq \exp\left(-\kappa t+48A(\kappa\alpha)^2\right),
\end{align}
where we used the inequality $(1+u)\leq e^u, \forall u\in\R$. 

Set $\DS \kappa=\frac{t}{4\alpha t+96A\alpha^2}$. Then, $\kappa\leq 1/(4\alpha)$ and the expression inside the exponent in the right hand side of \eqref{StepBern1} is bounded from above by $\DS -t^2/\left(4\alpha(2t+48A\alpha)\right)$. 

Now, note that $\E[T]\leq \E[T^2]^{1/2}=\E[Z]^{1/2}\leq (2A)^{1/2}\alpha$, hence, \eqref{StepBern1} implies

\begin{equation} \label{StepBern2}
	\PP\left[T\geq t+(2A)^{1/2}\alpha\right]\leq e^{-\frac{t^2}{4\alpha(2t+48A\alpha)}}.
\end{equation}
Since $Z=T^2$, we get
\begin{equation*} 
	\PP\left[Z\geq 2t^2+4A\alpha^2\right]\leq e^{-\frac{t^2}{4\alpha(2t+48A\alpha)}},
\end{equation*}
which implies Lemma \ref{Bernstein}, by taking $x=t^2/\alpha^2$. 
\end{proof}

Applying Lemma \ref{Bernstein} to the random variable $|X_{\hat M}-\theta|$ yields the second part of the theorem.

\subsection{Proof of Theorem \ref{ThmChangePointLB}}

\paragraph{Deterministic design}
The proof is straightforward for the case of the deterministic design (DD). Let $G_1=[0,0]$ and $G_2=[0,1/(2n)]$. Then $\PP_{G_1}=\PP_{G_2}$, since no point of the design falls in $G_1\triangle G_2$, and for any estimator $\hat G_n$, 
\begin{align*}
	2\sup_{G\in\mathcal S_0}\E_G\left[|\hat G_n\triangle G|\right] & \geq \E_{G_1}\left[|\hat G_n\triangle G_1|\right]+\E_{G_2}\left[|\hat G_n\triangle G_2|\right] \\
	& \geq \E_{G_1}\left[|\hat G_n\triangle G_1|+|\hat G_n\triangle G_2|\right] \\
	& \geq \E_{G_1}\left[|G_1\triangle G_2|\right] \mbox{ by the triangle inequality} \\
	& \geq \frac{1}{2n}. 
\end{align*}
\hfill $\blacksquare$

\paragraph{Random design}
The proof is only slightly different for the case of random design (RD). Let $G_1=[0,0]$ and $G_2=[0,1/(2n)]$. The key is to note equality of the conditional joint distributions $\displaystyle{\PP_{G_1}\left[\hspace{2mm}\cdot\hspace{2mm}\big|X_i>1/(2n), i=1,\ldots,n\right]}$ and $\displaystyle{\PP_{G_2}\left[\hspace{2mm}\cdot\hspace{2mm}\big|X_i>1/(2n), i=1,\ldots,n\right]}$. Indeed, if $Y_i=\xi_i$ and $Y'_i=\mathds 1(X_i\leq1/(2n)) +\xi_i$, for $i=1,\ldots,n$, then conditionally to the event $\{X_i>1/(2n), i=1,\ldots,n\}$, $\displaystyle{(Y_i)_{i=1,\ldots,n}=(Y'_i)_{i=1,\ldots,n}}$ and therefore have the same conditional law.
Hence, one can write:
\begin{align*}
	2\sup_{G\in\mathcal S_0}&\E_G\left[|\hat G_n\triangle G|\right] \\
	& \geq 2\sup_{G\in\mathcal S_0}\E_G\left[|\hat G_n\triangle G|\mathds 1(X_i>1/(2n), i=1,\ldots,n)\right] \\
	& \geq \E_{G_1}\left[|\hat G_n\triangle G_1|\mathds 1(X_i>1/(2n), i=1,\ldots,n)\right] \\
	& \hspace{10mm}+\E_{G_2}\left[|\hat G_n\triangle G_2|\mathds 1(X_i>1/(2n),  i=1,\ldots,n)\right] \\
	& = \Big(\E_{G_1}\left[|\hat G_n\triangle G_1|\big|X_i>1/(2n), i=1,\ldots,n\right] \\
	& \hspace{10mm}+\E_{G_2}\left[|\hat G_n\triangle G_2|\big|X_i>1/(2n), i=1,\ldots,n\right]\Big) \\
	& \hspace{20mm} \times \PP\left[X_1>1/(2n), i=1,\ldots,n\right] \\
	& = \Big(\E_{G_1}\left[|\hat G_n\triangle G_1|\big|X_i>1/(2n), i=1,\ldots,n\right] \\
	& \hspace{10mm}+\E_{G_1}\left[|\hat G_n\triangle G_2|\big|X_i>1/(2n), i=1,\ldots,n\right]\Big)\PP\left[X_1>1/(2n)\right]^n \\
	& \geq \E_{G_1}\left[|G_1\triangle G_2|\big|X_i>1/(2n), i=1,\ldots,n\right] \\
	& \hspace{10mm} \times \PP\left[X_1>1/(2n)\right]^n  \mbox{ by the triangle inequality} \\
	& \geq \frac{1}{2n}\left(1-\frac{1}{2n}\right)^n  \geq \frac{1}{4n}.\\
\end{align*}
\hfill $\blacksquare$

\subsection{Proof of Theorem \ref{theoremMu}}

Consider the event $E=\{|\hat G_n\triangle G|\leq \mu/2\}$. Since $\{X_i : i\in I_0\}$ is a deterministic and regular design, with step $2/n$, Theorem \ref{LSETheorem1} yields $\PP_G[E]\geq 1-C_1n^4e^{-C_2\mu n/2}$.
On the event $E$, $|\hat G_n\triangle G|<\mu\leq |G|$, yielding that $\hat G_n$ and $G$ must intersect. Thus, still on the event $E$,
\begin{equation*}
	|\hat G_n\triangle G|=|\hat b_n-b|+|\hat a_n-a|,
\end{equation*}
where we denoted by $G=[a,b]$ and $\hat G_n=[\hat a_n,\hat b_n]$.
Let $m=\frac{a+b}{2}$ and $\hat m_n=\frac{\hat a_n+\hat b_n}{2}$ be, respectively, the middle points of $G$ and $\hat G_n$. From now on, let us assume that $E$ holds. Then,
\begin{align}
	|\hat m_n-m| & \leq \frac{1}{2}(|\hat b_n-b|+|\hat a_n-a|) \nonumber \\
	& = \frac{1}{2}|\hat G_n\triangle G| \nonumber \\
	& \leq \frac{\mu}{4}. \label{middles}
\end{align}
Therefore, $\hat m_n\in G$ and, combining \eqref{middles} with the fact that $|G|\geq \mu$, 
\begin{equation} \label{Sep}
	\min(\hat m_n, 1-\hat m_n)\geq\frac{\mu}{4}.
\end{equation}
By \eqref{Sep}, $\# I_1^\epsilon \geq \frac{\mu n}{8}-1\geq \frac{\mu n}{16}$ for $n$ large enough, and for $\epsilon\in\{+,-\}$. 
Note that $\{X_i : i\in I_1^+\}$ (resp. $\{X_i : i\in I_1^-\}$) is a deterministic and regular design of the segment $[\hat m_n,1]$ with $b\in[\hat m_n,1]$ (resp. $[0,\hat m_n]$ with $a\in[0,\hat m_n]$), of cardinality greater or equal to $\frac{\mu n}{16}$, as we saw just before. Then, the change-points $a$ and $b$ are estimated as in Theorem \ref{ThmChangePointUB}, using the subsamples $\{(X_i,Y_i) : i\in I_1^+\}$ and $\{(X_i,Y_i) : i\in I_1^-\}$ respectively:
\begin{equation*}
	\PP_G\left[|\tilde a_n-a|\geq \frac{16y}{\mu n} \big |E\right]\leq C_0e^{-y/(8\sigma^2)}
\end{equation*}
and
\begin{equation*}
	\PP_G\left[|\tilde b_n-b|\geq \frac{16y}{\mu n}\big| E\right]\leq C_0e^{-y/(8\sigma^2)},
\end{equation*}
for all $y>0$.
Finally,
\begin{align*}
	\PP_G\left[|\tilde G_n\triangle G|\geq \frac{y}{n}\right] & \leq \PP_G\left[|\tilde G_n\triangle G|\geq \frac{y}{n}\big| E\right]+\PP_G[\bar{E}] \\
	& \leq 2C_0e^{-\mu y/(256\sigma^2)}+C_1n^4e^{-C_2\mu n/2},
\end{align*}
where $\bar E$ stands for the complement of the event $E$.\hfill $\blacksquare$

\subsection{Proof of Theorem \ref{theoremMuRD}}

Now, we consider the random design (RD). For simplicity's sake, we assume that $n$ is even, without loss of generality. As in the proof of Theorem \ref{theoremMu}, let $E$ be the event $\DS \{|\hat G_n\triangle G|\leq \mu/2\}$, where $\hat G_n$ is the preliminary estimator of $G$ based on the first subsample $\mathcal D_0$.

Let $k\in\{1,\ldots,n/2\}$. By a similar reasoning as in Lemma \ref{Lemma2}, conditional on the subsample $\mathcal D_0$ and on the event $\{\#I_1^+=k\}$, the vectors $(X_i,i\in I_1^+)$ and $(\xi_i,i\in I_1^+)$ are independent, the conditional distribution of $(\xi_i,i\in I_1^+)$ is the same as the unconditional distribution of $(\xi_i,i=1,\ldots,k)$ and the conditional distribution of $(X_i,i\in I_1^+)$ is equal to the distribution of the increasing reordering of $k$ i.i.d. random variables uniformly distributed in $[\hat m_n,1]$.
Let $t>0$. We first write that 
\begin{equation} \label{NewProof1}
	\PP[|\tilde G_n\triangle G|\geq t|\mathcal D_0]\leq \PP[|\tilde G_n\triangle G|\geq t|\mathcal D_0]\mathds 1(E)+\mathds 1\left(\bar E\right),
\end{equation}
where $\mathds 1(E)$ is one if $E$ is satisfied, 0 otherwise. Note that if $E$ is satisfied, then $|\tilde G_n\triangle G|=|\tilde a_n-a|+|\tilde b_n-b|$. Hence, \eqref{NewProof1} becomes
\begin{equation} \label{NewProof2}
	\PP[|\tilde G_n\triangle G|\geq t|\mathcal D_0]\leq \left(\PP[|\tilde a_n-a|\geq t/2|\mathcal D_0]+\PP[|\tilde b_n-b|\geq t/2|\mathcal D_0]\right)\mathds 1(E)+\mathds 1\left(\bar E\right).
\end{equation}
Let us bound the first term. 
\begin{equation} \label{NewProof2Bis}
	\PP[|\tilde a_n-a|\geq t/2|\mathcal D_0] = \sum_{k=0}^{n/2} \PP[|\tilde a_n-a|\geq t/2|\mathcal D_0,\#I_1^-=k]\PP[\#I_1^-=k|\mathcal D_0].
\end{equation}
Assume that the event $E$ is satisfied. Then, $\#I_1\geq N:=\#\{i\in I_1: X_i\geq \mu/4\}$, which is a binomial random variable with parameters $n/2,1-\mu/4$. Denote by $\lambda=(1-\mu/4)/4$. By Hausdorff inequality, for all integers $k\leq n\lambda$,
\begin{align*}
	\PP[\#I_1^-=k|\mathcal D_0] & \leq \PP[N\leq k] \\
	& \leq e^{-4n\lambda^2}.
\end{align*}
By the second part of Theorem \ref{ThmCPRD}, for all $k\geq 1$,
\begin{equation} \label{NewProof3}
	\PP[|\tilde a_n-a|\geq t/2|\mathcal D_0,\#I_1^-=k] \leq \begin{cases}
		1 \mbox{ if } t<\frac{A_1\sigma^4}{k} \\
		\Psi\left(\frac{kt}{\sigma^2}-A_1\sigma^2\right) \mbox{ otherwise},
	\end{cases}
\end{equation}
where we denoted by $\DS \Psi(x)=e^{-\frac{x}{A_2\sqrt x +A_3\sigma^2}}$.
In the sequel, we assume that $\DS t\geq (1+\eta)\frac{A_1\sigma^4}{\lambda n}$, where $\eta>0$ is a constant that we will choose later. Then, for all $k>n\lambda$, it holds that $t\geq \frac{A_1\sigma^4}{k}$ and we can use the inequality of the second case in \eqref{NewProof3}. Hence, it follows from \eqref{NewProof2Bis} that

\begin{equation} \label{NewProof4}
\PP[|\tilde a_n-a|\geq t/2|\mathcal D_0] \leq \sum_{k=0}^{\lambda n} e^{-4n\lambda^2} + \sum_{k=\lambda n}^{n/2}\Psi\left(\frac{kt}{\sigma^2}-A_1\sigma^2\right)\PP[\#I_1^-=k|\mathcal D_0].
\end{equation}

Now, we choose the constant $\DS \eta=\frac{A_3\sigma^2}{A_2A_1}$, so that, for $k\geq \lambda n$ and $x=\frac{kt}{\sigma^2}-A_1\sigma^2$, $\DS \Psi(x)\leq e^{-\frac{\sqrt x}{2A_3\sigma^2}}$. In addition, still for $k\geq\lambda n$, $\DS x\geq \frac{\eta}{1+\eta}\frac{kt}{\sigma^2}$. Hence, \eqref{NewProof4} becomes

\begin{align} \label{NewProof5}
	\PP[|\tilde a_n-a|\geq t/2|\mathcal D_0] & \leq \lambda ne^{-4n\lambda^2}+\sum_{k=\lambda n}^{\infty}e^{-C_1\sqrt{kt}}\PP[\#I_1^-=k|\mathcal D_0] \nonumber \\
	& \leq \lambda ne^{-4n\lambda^2}+e^{-C_1\sqrt{\lambda nt}}
\end{align}
where $C_1$ is a positive constant that depends on $\sigma^2$ only.

The same reasoning yields, still under the assumption that the event $E$ holds,
\begin{align} \label{NewProof6}
	\PP[|\tilde b_n-b|\geq t/2|\mathcal D_0] \leq \lambda_1 ne^{-4n\lambda_1^2}+e^{-C_1\sqrt{\lambda_1 nt}},
\end{align}
for all $t\geq (1+\eta)A_1\sigma^4/(\lambda_1 n)$, where $\lambda_1=\mu/16$. Finally, \eqref{NewProof2} becomes

\begin{equation} \label{NewProof7}
	\PP[|\tilde G_n\triangle G|\geq t|\mathcal D_0]\leq B_1e^{-B_2n}+2e^{-C_3\sqrt{nt}}+\mathds 1\left(\bar E\right),
\end{equation}
for all $t\geq C_4/n$, where $C_1, C_2, C_3$ and $C_4$ are 
positive constants that depend on $\mu$ and $\sigma^2$ only. Since Theorem \ref{LSETheorem1} yields $\PP_G[E]\geq 1-C_1n^4e^{-C_2\mu n/2}$, taking the expectation in \eqref{NewProof7} with respect to $\mathcal D_0$ gives the desired result. \hfill $\blacksquare$

\bibliographystyle{alpha}
\bibliography{Biblio}

\end{document}